\documentclass[english]{article}
\usepackage[T1]{fontenc}
\usepackage[latin9]{inputenc}
\usepackage{geometry}
\usepackage{verbatim}
\usepackage{amsthm}
\usepackage{amsmath}
\usepackage{amssymb}
\usepackage{esint}

\makeatletter
\theoremstyle{plain}
\newtheorem{thm}{\protect\theoremname}
  \theoremstyle{plain}
  \newtheorem{lem}{\protect\lemmaname}
  \theoremstyle{plain}
  \newtheorem{prop}{\protect\propositionname}

\makeatother

\usepackage{babel}
  \providecommand{\lemmaname}{Lemma}
  \providecommand{\propositionname}{Proposition}
\providecommand{\theoremname}{Theorem}

\begin{document}
\global\long\def\CC{\mathbb{C}}
\global\long\def\cW{\mathcal{W}}

\global\long\def\RR{\mathbb{R}}
\global\long\def\Lie{\mathrm{Lie}}

\global\long\def\cW{\mathcal{W}}
\global\long\def\GL{\mathrm{GL}}

\global\long\def\cU{\mathcal{U}}

\title{A note on Kirillov model for representations of $\GL_{n}(\CC)$}

\author{Alexander Kemarsky}
\maketitle
\begin{abstract}
Let $G=\GL_{n}(\CC)$ and $1\ne\psi:\CC\to\CC^{\times}$ be an additive
character. Let $U$ be the subgroup of upper triangular unipotent
matrices in $G$. Denote by $\theta$ the character $\theta:U\to\CC$
given by
\[
\theta(u):=\psi(u_{1,2}+u_{2,3}+...+u_{n-1,n}).
\]
Let $P$ be the mirabolic subgroup of $G$ consisting of all matrices
in $G$ with the last row equal to $(0,0,...,0,1)$. We prove that
if $\pi$ is an irreducible generic representation of $\GL_{n}(\CC)$
and $\cW(\pi,\psi)$ is its Whittaker model, then the space $\{f|_{P}:P\to\CC:\, f\in\cW(\pi,\psi)\}$
contains the space of infinitely differentiable functions $f:P\to\CC$
which satisfy $f(up)=\psi(u)f(p)$ for all $u\in U$ and $p\in P$
and which have a compact support modulo $U$. A similar result was
proven for $\GL_{n}(F)$, where $F$ is a $p$-adic field by Gelfand-Kazhdan
\cite{gelfand1975representations}, and for $\GL_{n}(\RR)$ by Jacquet
\cite{jacquet2010distinction}.
\end{abstract}
Let $F$ be the field $\RR$ or $\CC$.%
{} Let $G_{n}(F)$ be the group $\GL_{n}(F)$, let $P(F)$ be the mirabolic
subgroup in $G_{n}(F)$ consisting of matrices in $G_{n}(F)$ with
the last row equal to $(0,0,...,0,1)$. Let $U_{n}(F)$ ( $\overline{U_{n}(F)}$
respectively) be the unipotent subgroups consisting of upper triangular
unipotent matrices ( lower triangular unipotent matrices respectively)
in $G_{n}(F)$. We fill fix a field $F$ and will abbreviate $G$
for the group $G(F)$. Let $\psi:F\to\CC^{\times}$ be a non-trivial
additive character of $F$ and denote by $\theta_{n}$ a character
on $U_{n}$ given by
\[
\theta_{n}(x):=\psi(x_{12}+x_{23}+...+x_{n-1,n})
\]
for $x\in U$. Let $\pi$ be a unitary irreducible representation
of $G_{n}$ on a Hilbert space \global\long\def\cH{\mathcal{H}}
$\cH$ with norm $||\cdot||.$ We let \global\long\def\cV{\mathcal{V}}
$\cV$ be the space of the $G_{n}$ smooth vectors, equipped with
the topology defined by the semi-norms
\[
v\to||d\pi(X)v||
\]
with $X\in\cU(G_{n})$. We let $\cV'$ be the topological conjugate
dual. We have inclusions
\[
\cV\subset\cH\subset\cV'.
\]
The positive-definite scalar product on $\cV\times\cV$ extends to
$\cV\times\cV'$. Recall $\pi$ is said to be generic if there is
an element $\lambda\ne0$ of $\cV'$ such that
\[
\left(\pi(u)v,\lambda\right)=\theta_{n}(u)\left(v,\lambda\right)
\]
for all $v\in\cV$ and $u\in U_{n}.$ Up to a scalar factor, the vector
$\lambda$ is unique \cite{shalika1974multiplicity}. For every $v\in\cV$
we define a function $W_{v}$ on $G_{n}$ by
\[
W_{v}(g)=(\pi(g)v,\lambda).
\]
We let $\cW(\pi,\psi)$ be the Whittaker model of $\pi$, that is,
the space spanned by the functions
\[
g\mapsto(\pi(g)v,\lambda)\,,
\]
where $v$ is in $\cV$. We identify $\cV$ and $\cW(\pi,\psi).$
Then the scalar product induced by $\cH$ on $\cV$ is a scalar multiple
of the scalar product defined by the convergent integral
\[
(v_{1},v_{2}):=\int_{U_{n-1}\backslash G_{n-1}}W_{v_{1}}\left(\begin{array}{cc}
g & 0\\
0 & 1
\end{array}\right)\overline{W_{v_{2}}\left(\begin{array}{cc}
g & 0\\
0 & 1
\end{array}\right)}dg,
\]
see \cite{jacquet1981euler}. We may assume the scalar product on
$\cV$ is equal to this convergent integral. %
{} \\
We denote by $C_{c}^{\infty}(\theta_{n-1},G_{n-1})$ the space of
smooth functions $f:G_{n-1}\to\CC$ which are compactly supported modulo $U_{n-1}$, 
such that $f(ug)=\theta_{n}(u)f(g)$ for all $u\in U_{n-1},\, g\in G_{n-1}.$
\\
In this note we prove the following theorem for $F=\CC$.
\begin{thm}
\label{thm 1} Let $\pi$ be a generic unitary irreducible representation
of $G_{n}(:=G_{n}(\CC))$. Given $\phi\in C_{c}^{\infty}(\theta_{n-1},G_{n-1})$
there is a unique $W_{\phi}\in\cV$ such that, for all $g\in G_{n-1},$
\[
W_{\phi}\left(\begin{array}{cc}
g & 0\\
0 & 1
\end{array}\right)=\phi(g).
\]
Furthermore, the map $\phi\to W_{\phi}$ is continuous.
\end{thm}
An analogous result was proven for $G_{n}(F)$, $F$ a $p$-adic field
by Gelfand-Kazhdan \cite{gelfand1975representations}. Recently it
was proven by Jacquet for $G_{n}(\RR)$ \cite{jacquet2010distinction}.
Our proof for $G_{n}(\CC)$ is similar to Jacquet proof. We sketch
here the main steps of the proof. There is only one technical lemma
that should be reproved for $G_{n}(\CC).$ \\
For a Lie algebra $\mathfrak{g}$ over $\CC$ we denote by $\cU(\mathfrak{g)}$
its universal enveloping algebra. For a Lie group $G$ we denote by
$\Lie(G)$ its Lie algebra (over $\CC)$, by $\cU(G)$ its universal
enveloping algebra, by $Z(\cU(G))$ the center of $\cU(G)$. Denote
the image of $X\in\Lie(G)$ under the natural embedding in $\cU(G)$
by $D_{X}$ . We will identify $\Lie(G)$ with its image in $\cU(G)$.
Note that $D_{X}$ can be realized as the linear differential operator
on the space $C_{c}^{\infty}(G)$ corresponding to derivative in the
direction $X$.\\
Let us denote the elementary matrix with $1$ in the $(i,j)$ place
and $0$ in all other places by $E_{ij}.$ Let us denote $D_{E_{ij}}:=D_{(i,j)}.$
Denote the elementary matrix with $\sqrt{-1}$ in the place $(i,j)$
and $0$ in all other places by $D_{\sqrt{-1}(i,j)}$. The set
\[
\left\{ D_{ij,}D_{\sqrt{-1}(i,j)}:1\le i,j\le n\right\}
\]
is a basis of $\Lie(G_{n}(\CC))$. Observe that there are two natural
embeddings of Lie algebras
\[
i_{1},i_{2}:\Lie(G_{n}(\RR))\to\Lie(G_{n}(\CC))
\]
given by
\[
i_{1}(D_{X})=\frac{1}{2}\left(D_{X}+\sqrt{-1}D_{\sqrt{-1}X}\right),\, i_{2}(D_{X})=\frac{1}{2}\left(D_{X}-\sqrt{-1}D_{\sqrt{-1}X}\right).
\]
 We have
\[
\Lie\left(G_{n}(\CC)\right)=i_{1}\left(\Lie\left(G_{n}(\RR)\right)\right)\oplus i_{2}\left(\Lie\left(G_{n}(\RR)\right)\right).
\]
That is, the complex vector space $\Lie\left(G_{n}(\CC)\right)$ is
a direct sum of the vector spaces
\[
i_{1}\left(\Lie\left(G_{n}(\RR)\right)\right),i_{2}\left(\Lie\left(G_{n}(\RR)\right)\right),
\]
and for all $X,Y\,\in\Lie\left(G_{n}\left(\RR\right)\right)$ we have
$[i_{1}(X),i_{2}(Y)]=0.$\\
The following lemma is the key lemma in the proof of Theorem \ref{thm 1}
for $G_{n}(\RR)$.
\begin{lem}
Consider a $\Lie(G_{n}(F))$ module $V$ and a vector $0\ne\lambda\in V$.
Suppose $D_{ij}\lambda=0$ for every pair of indices $1\le i,j\le n$
such that $j>i+1$. Suppose also that for every $1\le i\le n-1$ there
is a constant $c_{i}\ne0$ such that $D_{i,(i+1)}\lambda=c_{i}\lambda$.
If $F=\CC$ suppose also that $D_{\sqrt{-1}i,j}\lambda=0$ for every
pair of indices $1\le i,j\le n$ such that $j>i+1$ and for every
$1\le i\le n-1$ there is a constant $c_{\sqrt{-1}i}\ne0$ such that
$D_{\sqrt{-1}(i,i+1)}\lambda=c_{\sqrt{-1}i}\lambda.$ Then
\begin{gather*}
\Lie\left(\overline{U_{n}(F)}\right)\lambda\subset\cU(G_{n-1}(F))Z(\cU(G_{n}(F))\lambda,\\
\cU\left(\overline{U_{n}(F)}\right)\lambda\subset\cU(G_{n-1}(F))Z(\cU(G_{n}(F))\lambda,\\
\cU\left(\overline{G_{n}(F)}\right)\lambda\subset\cU(G_{n-1}(F))Z(\cU(G_{n}(F))\lambda.
\end{gather*}
\end{lem}
\begin{proof}
We will use the corresponding result for the case $\Lie(G_{n}(\RR))$
and deduce from it the result for $\Lie(G_{n}(\CC))$. Note that it
is enough to prove the first inclusion, as the second inclusion and
the third inclusion follows from the first by an application of theorem
of Poincare-Birkhoff-Witt. See \cite{jacquet2010distinction}, the
first lines of the proof of Lemma 3. \\
To prove the first inclustion we have to prove that
\[
D_{X}\lambda\in\cU(G_{n-1}(F))Z(\cU(G_{n}(F))\lambda
\]
for
\[
X\in\left\{ E_{n1},E_{n2},...,E_{nn}\right\} \cup\left\{ E_{\sqrt{-1}(n,1)},E_{\sqrt{-1}(n,2)},...,E_{\sqrt{-1}(n,n)}\right\} .
\]
By the corresponding result for $\Lie(G_{n}(\RR))$ we know that for
the Lie algebras $i_{1}\left(\Lie\left(G_{n-1}(\RR)\right)\right)$
and $i_{2}\left(\Lie\left(G_{n-1}(\RR)\right)\right)$ and $X\in\left\{ E_{n1},E_{n2},...,E_{nn}\right\} $
we have
\[
D_{i_{1,2}(X)}\lambda\in\cU\left(i_{1,2}\left(\Lie\left(G_{n-1}(\RR)\right)\right)\right)Z\left(\cU\left(i_{1,2}\left(\Lie\left(G_{n-1}(\RR)\right)\right)\right)\right)\lambda\subset\cU\left(G_{n-1}(\CC)\right)Z\left(\cU\left(G_{n-1}(\CC)\right)\right)\lambda.
\]
Thus,
\[
\frac{1}{2}\left(D_{n,r}\lambda+\sqrt{-1}D_{\sqrt{-1}n,r}\lambda\right),\,\frac{1}{2}\left(D_{n,r}-\sqrt{-1}D_{\sqrt{-1}n,r}\lambda\right)\in\cU\left(G_{n-1}(\CC)\right)Z\left(\cU\left(G_{n-1}(\CC)\right)\right)\lambda.
\]
 By adding and substracting the two terms on the left hand-side of
the last formula we obtain
\[
D_{n,r}\lambda,D_{\sqrt{-1}n,r}\lambda\in\cU\left(G_{n-1}(\CC)\right)Z\left(\cU\left(G_{n-1}(\CC)\right)\right)\lambda.
\]
 The lemma is proved.
\end{proof}
For the convenience of a reader we rewrite here only formulations
of the lemmas which are needed to prove Theorem \ref{thm 1}. All
the proofs are identical to that written in \cite{jacquet2010distinction},
so we see no need to repeat the proofs here.\\
From now on $F=\CC$. Consider the space $\cV_{n-1}$ of $G_{n-1}$
smooth vectors in $\cH$.
\begin{lem}
We have continuous inclusions
\[
\cV\subset\cV_{n-1}\subset\cH\subset\cV'_{n-1}\subset\cV'.
\]

\end{lem}
~
\begin{lem}
The vector $\lambda$ which is appriori in $\cV'$ belongs to $\cV'_{n-1}$.
\end{lem}
Given $\phi\in C_{c}^{\infty}(G_{n-1})$, we set
\[
u_{\phi}:=\int_{g\in G_{n-1}}\phi(g)\pi\left(\begin{array}{cc}
g & 0\\
0 & 1
\end{array}\right)\lambda\, dg.
\]

\begin{lem}
The vector $u_{\phi}$ belongs to $\cV_{n-1}$, in particular to $\cH$.
The map
\[
\phi\to u_{\phi},\, C_{c}^{\infty}(G_{n-1})\to\cH
\]
is continuous.
\end{lem}
~
\begin{lem}
\label{lem: smooth vecs}For every $X\in\cU(G_{n})$ and every $\phi\in C_{c}^{\infty}(G_{n-1})$
the vector $u_{\phi}$, which, a priori, is in $\cV'$, is in fact
in $\cH.$ Moreover, there is a continuous semi-norm $\mu$ on $C_{c}^{\infty}(G_{n-1})$
such that, for every $\phi\in C_{c}^{\infty}(G_{n-1}),$
\[
||d\pi(X)u_{\phi}||\le\mu(\phi).
\]

\end{lem}
We note that Lemma 1 is used to prove the last Lemma.
\begin{lem}
Suppose $v_{0}\in\cH$ is a vector such that for every $X\in\cU(G_{n})$
the vector $d\pi(X)v_{0}$, which a priori is in $\cV'$, is in fact
in $\cH$. Then $v_{0}$ is in $\cV$.\end{lem}
\begin{prop}
For every $\phi\in C_{c}^{\infty}(G_{n-1})$ the vector
\[
u_{\phi}:=\int_{g\in G_{n-1}}\phi(g)\pi\left(\begin{array}{cc}
g & 0\\
0 & 1
\end{array}\right)\lambda\, dg
\]
is in $\cV$. Furthermore, the map
\[
\phi\mapsto u_{\phi},\, C_{c}^{\infty}(G_{n-1})\to\cV
\]
is continuous.
\end{prop}
~
\begin{lem}
If $v=u_{\phi}$ then
\[
W_{v}\left(\begin{array}{cc}
g & 0\\
0 & 1
\end{array}\right)=\phi_{0}(g),
\]
where
\[
\phi_{0}(g)=\int_{U_{n-1}}\phi((gu)^{-1})\overline{\theta_{n-1}}(u)du.
\]

\end{lem}
I would like to thank Omer Offen and Erez Lapid for posing me this
question, and Dmitry Gourevitch and Herv\'e Jacquet for their interest
in this project.\\
The research was supported by ISF grant No. 1394/12.
\small
\bibliographystyle{plain}
\bibliography{kirillov_refs}

\end{document}